\documentclass[11pt]{amsart}

\usepackage[utf8]{inputenc}
\usepackage[T1]{fontenc}
\usepackage[english]{babel} 
\usepackage{mathrsfs}
\usepackage{tikz}
\usepackage{color}

\theoremstyle{plain} 
\newtheorem{thm}{Theorem} 
\newtheorem{lem}[thm]{Lemma} 
\newtheorem{cor}[thm]{Corollary}

\theoremstyle{definition} 
\newtheorem*{defn}{Definition}

\DeclareMathOperator{\mre}{Re}

\begin{document} 
\title[Sharp norm estimates for composition operators]{Sharp norm estimates for composition operators and Hilbert-type inequalities}
\date{\today} 

\author{Ole Fredrik Brevig} 
\address{Department of Mathematical Sciences, Norwegian University of Science and Technology (NTNU), NO-7491 Trondheim, Norway} 
\email{ole.brevig@math.ntnu.no}

\thanks{The author is supported by Grant 227768 of the Research Council of Norway.}


\begin{abstract}
	Let $\mathscr{H}^2$ denote the Hardy space of Dirichlet series $f(s) = \sum_{n\geq1} a_n n^{-s}$ with square summable coefficients and suppose that $\varphi$ is a symbol generating a composition operator on $\mathscr{H}^2$ by $\mathscr{C}_\varphi(f) = f \circ \varphi$. Let $\zeta$ denote the Riemann zeta function and $\alpha_0=1.48\ldots$ the unique positive solution of the equation $\alpha\zeta(1+\alpha)=2$. We obtain sharp upper bounds for the norm of $\mathscr{C}_\varphi$ on $\mathscr{H}^2$ when $0<\mre\varphi(+\infty)-1/2 \leq \alpha_0$, by relating such sharp upper bounds to the best constant in a family of discrete Hilbert-type inequalities.
\end{abstract}

\maketitle

\section{Introduction} \label{sec:intro}
Let $0<\alpha<\infty$. The main object of study in the present paper is the family of discrete bilinear forms
\begin{equation} \label{eq:bform}
	B_\alpha(a,b) := \sum_{m=1}^\infty \sum_{n=1}^\infty a_m b_n \frac{(mn)^{\alpha-1/2}}{[\max(m,n)]^{2\alpha}}.
\end{equation}
Let $\|B_\alpha\|$ denote the norm of the bilinear form \eqref{eq:bform} on $\ell^2\times\ell^2$, that is the smallest positive number $C_\alpha$ such that
\[|B_\alpha(a,b)| \leq C_\alpha \|a\|_{\ell^2} \|b\|_{\ell^2}\]
holds for every pair of sequences $a,b\in\ell^2$. Our interest in $\|B_\alpha\|$ stems from its connection to sharp norm estimates for composition operators. For the moment, we postpone the discussion of this connection and focus on \eqref{eq:bform}. Let
\[\zeta(s) := \sum_{n=1}^\infty n^{-s}\]
denote the Riemann zeta function and let $\alpha_0=1.48\ldots$ denote the unique positive solution of the equation $\alpha\zeta(1+\alpha)=2$. We have been unable to compute $\|B_\alpha\|$ for every $0<\alpha<\infty$, but we can prove the following result.

\begin{thm} \label{thm:Balpha}
	For $0<\alpha<\infty$,
	\[\max\left(\frac{2}{\alpha}\,,\,\zeta(1+2\alpha)\right) \leq \|B_\alpha\| \leq \max\left(\frac{2}{\alpha}\,,\,\zeta(1+\alpha)\right).\]
	In particular, if $0<\alpha \leq \alpha_0 = 1.48\ldots$, then $\|B_\alpha\|=2/\alpha$.
\end{thm}

See Figure~\ref{fig:plots} for a graph of the functions $2/\alpha$, $\zeta(1+\alpha)$ and $\zeta(1+2\alpha)$ on the interval $[1,2]$, where both intersections occur. Theorem~\ref{thm:Balpha} is presented as a combination of several results. The upper bound comes from Lemma~\ref{lem:Sasup}. The first part of the lower bound is Lemma~\ref{lem:2alower}, while the second is obtained by the point estimate \eqref{eq:peval} and Theorem~\ref{thm:bform}. 

Our approach to the bilinear form \eqref{eq:bform} and to Theorem~\ref{thm:Balpha} is classical. We will exploit that the positive and symmetric kernels
\begin{equation} \label{eq:Ka}
	K_\alpha(x,y) = \frac{(xy)^{\alpha-1/2}}{[\max(x,y)]^{2\alpha}}
\end{equation}
enjoy the homogeneity property $K_\alpha(\lambda x,\lambda y) = \lambda^{-1}K_\alpha(x,y)$ for $x,y,\lambda>0$. Hence \eqref{eq:bform} is a Hilbert--type bilinear form as studied by Hardy, Littlewood and P\'olya \cite[Ch.~IX]{HLP}. In fact, $\alpha=1/2$ in Theorem~\ref{thm:Balpha} is \cite[Thm.~341]{HLP}. Note also that the result in \cite{Brevig17} covers the case $\alpha=1$.

We first investigate the continuous version of the discrete bilinear form \eqref{eq:bform}. As expected, it is easy to show that the norm of the continuous bilinear form is $2/\alpha$ for every $0<\alpha<\infty$ (see Theorem~\ref{thm:Ha}). Inspired by \cite{HLP}, we aim to use discretization to obtain a sharp result for \eqref{eq:bform}. This approach is successful when $0<\alpha\leq1$. In fact, the upper bound in Theorem~\ref{thm:Balpha} can be deduced directly from \cite[Thm.~318]{HLP} in this range.

A phase change occurs at $\alpha=1$, and the discretization argument gives here an upper bound which (from its application) clearly cannot be sharp. The main point of Theorem~\ref{thm:Balpha} is therefore that the upper bound obtained by discretization for $0<\alpha\leq1$ extends beyond the phase change at $\alpha=1$ to (at least) $\alpha_0=1.48\ldots$. Note that the upper bound $2/\alpha$ cannot hold when $\alpha\geq2$, since it would contradict that $\|B_\alpha\|>1$. In fact, we have verified that the upper bound $2/\alpha$ fails when $\alpha\geq1.7$ (see Section~\ref{sec:remarks}).


\begin{figure}[t]
	\begin{tikzpicture}[scale=1.15]
		\draw (0,0) -- coordinate (x axis mid) (10,0);
		\draw (0,0) -- coordinate (y axis mid) (0,5);
    	\draw (2,2pt) -- (2,-2pt) node[anchor=north] {1.2};
    	\draw (4,2pt) -- (4,-2pt) node[anchor=north] {1.4};
    	\draw (6,2pt) -- (6,-2pt) node[anchor=north] {1.6};
		\draw (8,2pt) -- (8,-2pt) node[anchor=north] {1.8};
    	\draw (2pt,1) -- (-2pt,1) node[anchor=east] {1.2};
    	\draw (2pt,2) -- (-2pt,2) node[anchor=east] {1.4};
    	\draw (2pt,3) -- (-2pt,3) node[anchor=east] {1.6};
		\draw (2pt,4) -- (-2pt,4) node[anchor=east] {1.8};
		\draw[thick,smooth,samples=100,domain=0:10,color=black] plot(\x,{5*(2/(1+\x/10)-1)}) ; 
		\draw[thick,smooth,samples=100,domain=0:10,color=blue] plot(\x,{5*(2^(-(1+\x/10)-1)+3^(-(1+\x/10)-1)+4^(-(1+\x/10)-1)+5^(-(1+\x/10)-1)+6^(-(1+\x/10)-1)+7^(-(1+\x/10)-1)+8^(-(1+\x/10)-1)+9^(-(1+\x/10)-1)+9.5^(-(1+\x/10))/(1+\x/10))}) ;
		\draw[thick,smooth,samples=100,domain=0:10,color=red] plot(\x,{5*(2^(-2*(1+\x/10)-1)+3^(-2*(1+\x/10)-1)+4^(-2*(1+\x/10)-1)+5^(-2*(1+\x/10)-1)+6^(-2*(1+\x/10)-1)+7^(-2*(1+\x/10)-1)+8^(-2*(1+\x/10)-1)+9^(-2*(1+\x/10)-1)+9.5^(-2*(1+\x/10))/(2*(1+\x/10)))}) ;
	\end{tikzpicture}
	\label{fig:plots}
	\caption{{\color{black}$2/\alpha$}, {\color{blue}$\zeta(1+\alpha)$} and {\color{red}$\zeta(1+2\alpha)$} for $1\leq \alpha \leq 2$.}
\end{figure}
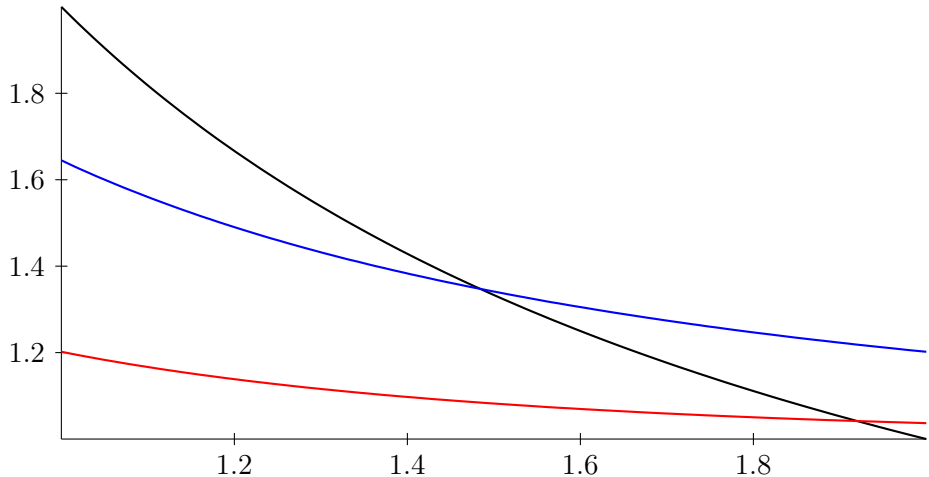


To set the stage for the discussion of the relationship between $B_\alpha$ and composition operators, let $H^2(\mathbb{D})$ denote the Hardy space of the unit disc $\mathbb{D}:=\left\{z\,:\,|z|<1\right\}$, consisting of analytic functions $F(z) = \sum_{k\geq0} a_k z^k$ with square summable coefficients, 
\[\|F\|_{H^2(\mathbb{D})} := \left(\sum_{k=0}^\infty |a_k|^2 \right)^{1/2} = \lim_{r\to 1^-}\left(\int_0^{2\pi} |F(re^{i\theta})|^2\,\frac{d\theta}{2\pi}\right)^\frac{1}{2}.\]
For the following basic facts about Hardy spaces of the unit disc and their composition operators, we refer to \cite[Ch.~11]{Zhu}. Suppose that $\phi\colon\mathbb{D}\to\mathbb{D}$ is analytic. Littlewood's subordination principle \cite{Littlewood25} gives the following upper bound for the norm of the composition operator defined by $\mathscr{C}_\phi(F) = F\circ\phi$:
\begin{equation} \label{eq:upperbound}
	\|\mathscr{C}_\phi\|_{H^2(\mathbb{D})\to H^2(\mathbb{D})} \leq \sqrt{\frac{1+|\phi(0)|}{1-|\phi(0)|}}.
\end{equation}
From the functional of point evaluation, we obtain the following lower bound: 
\begin{equation} \label{eq:lowerbound}
	\|\mathscr{C}_\phi\|_{H^2(\mathbb{D})\to H^2(\mathbb{D})}\geq \sqrt{\frac{1}{1-|\phi(0)|^2}}.
\end{equation}
Both \eqref{eq:upperbound} and \eqref{eq:lowerbound} are sharp for any value $\phi(0)=w\in\mathbb{D}$. Indeed, for the lower bound, take $\phi(z)=w$, and for the upper bound, take the M\"obius transform
\[\phi(z) = \frac{w-z}{1-\overline{w}z}.\]
In general, the computation of $\|\mathscr{C}_\phi\|_{H^2(\mathbb{D})\to H^2(\mathbb{D})}$ is difficult, but it follows from \eqref{eq:upperbound} and \eqref{eq:lowerbound} that the composition operator $\mathscr{C}_\phi$ is a contraction on $H^2(\mathbb{D})$ if and only if $\phi(0)=0$.

We will use Theorem~\ref{thm:Balpha} to obtain sharp norm estimates for composition operators on $\mathscr{H}^2$, the Hardy space of Dirichlet series $f(s) = \sum_{n\geq1}a_n n^{-s}$ with square summable coefficients, 
\[\|f\|_{\mathscr{H}^2} := \left(\sum_{n=1}^\infty |a_n|^2\right)^{1/2}.\]
The basic properties of $\mathscr{H}^2$ can be found in \cite{HLS97,QQ13}. Using the Cauchy--Schwarz inequality, it is easy to see that $\mathscr{H}^2$ is a space of absolutely convergent Dirichlet series in the half-plane $\mathbb{C}_{1/2}$, where
\[\mathbb{C}_\theta := \left\{s \,:\, \mre(s)>\theta \right\}.\]
To see that $\mathbb{C}_{1/2}$ is generally the largest domain of convergence for functions in $\mathscr{H}^2$, consider $f(s) = \zeta(1/2+\varepsilon+s)$ for $\varepsilon>0$.

The study of composition operators on $\mathscr{H}^2$ was initiated by Gordon and Hedenmalm in their pioneering paper \cite{GH99} (see also~\cite{Queffelec15,QS15}), where they proved that an analytic function $\varphi\colon\mathbb{C}_{1/2}\to\mathbb{C}_{1/2}$ generates a composition operator on $\mathscr{H}^2$ if and only if it is a member of the following class.

\begin{defn}
	The \emph{Gordon--Hedenmalm} class, $\mathscr{G}$, consists of symbols of the form
	\[\varphi(s) = c_0 s + \sum_{n=1}^\infty c_n n^{-s} =: c_0 s + \varphi_0(s),\]
	where $c_0$ is a non-negative integer. The series $\varphi_0$ converges uniformly in $\mathbb{C}_\varepsilon$ for every $\varepsilon>0$ and satisfies the following mapping properties:
	\begin{itemize}
		\item[(a)] If $c_0=0$, then $\varphi_0(\mathbb{C}_0)\subset\mathbb{C}_{1/2}$.
		\item[(b)] If $c_0\geq1$ then $\varphi_0 \equiv 0 $ or $\varphi_0(\mathbb{C}_0)\subset\mathbb{C}_0$.
	\end{itemize}
\end{defn}

Observe that the case (b), which is when $\varphi(+\infty)=+\infty$, corresponds to $\phi(0)=0$ considered above. Indeed, it was shown in \cite{GH99} that composition operators generated by symbols with $c_0\geq1$ are contractive. Considering the case (a), the functional of point evaluation for $\mathscr{H}^2$ (see~\cite{GH99,HLS97}) gives that
\begin{equation} \label{eq:peval}
	\|\mathscr{C}_{\varphi}\| \geq \sqrt{\zeta(2\mre(c_1))},
\end{equation}
since $c_1 = \varphi(+\infty)$. The lower bound \eqref{eq:peval} is sharp for $\varphi(s)=c_1$. Our initial motivation for studying $B_\alpha$ was to investigate sharp upper bounds analogous to \eqref{eq:upperbound} for composition operators on $\mathscr{H}^2$ for the case (a), something which has been unresolved since \cite{GH99}. Here is the main result of the present paper. 
\begin{thm} \label{thm:bform}
	Fix $w \in \mathbb{C}_{1/2}$ and let $\alpha=\mre(w)-1/2$. Then
	\[\sup_{\substack{\varphi \in \mathscr{G} \\ \varphi(+\infty)=w}} \|\mathscr{C}_\varphi\| = \sqrt{\|B_\alpha\|}.\]
\end{thm}

The proof of Theorem~\ref{thm:bform} is obtained by combining observations and ideas from \cite{BB16,Brevig17,GH99}. Our desired sharp upper bound for composition operators is easily deduced from Theorem~\ref{thm:Balpha} and Theorem~\ref{thm:bform}.

\begin{cor} \label{cor:Cphileq}
	Let $\alpha_0=1.48\ldots$ denote the unique positive solution to the equation $\alpha\zeta(1+\alpha)=2$. Suppose that $\varphi$ is in $\mathscr{G}$ with $c_0=0$ and that  $0<\mre(c_1)-1/2 \leq \alpha_0$. Then
	\begin{equation} \label{eq:Cphileq}
		\|\mathscr{C}_\varphi\| \leq \sqrt{\frac{2}{\mre(c_1)-1/2}},
	\end{equation}
	Moreover, for every $0<\mre(c_1)-1/2\leq \alpha_0$, there are $\varphi$ in $\mathscr{G}$ attaining \eqref{eq:Cphileq}.
\end{cor}

The present paper is organized as follows. The proof of Theorem~\ref{thm:bform} is presented in Section~\ref{sec:proof}. The following two sections are devoted to the bilinear form \eqref{eq:bform} and Theorem~\ref{thm:Balpha}. In Section~\ref{sec:riemann} we follow \cite{HLP} and investigate the continuous version of \eqref{eq:bform}. Here we also obtain the lower bound in Theorem~\ref{thm:Balpha} for all $0<\alpha<\infty$ and the upper bound when $0<\alpha\leq1$ and $\alpha\geq3$. Section~\ref{sec:lemma} contains the proof of the upper bound in Theorem~\ref{thm:Balpha} in the most intricate cases $1<\alpha<2$ and $2\leq \alpha < 3$. Finally, Section~\ref{sec:remarks} contains a few remarks pertaining to the relationship between composition operators on $H^2(\mathbb{D})$ and $\mathscr{H}^2$. Also found in Section~\ref{sec:remarks} are some observations regarding Theorem~\ref{thm:Balpha} for $\alpha>\alpha_0$ and two interesting or appealing special cases of \eqref{eq:bform}.

\section{Proof of Theorem~\ref{thm:bform}} \label{sec:proof}
For fixed $0<\alpha<\infty$, the conformal map
\begin{equation} \label{eq:Ta}
	\mathcal{T}_\alpha(z) := \alpha\frac{1-z}{1+z}
\end{equation}
sends $\mathbb{D}$ to $\mathbb{C}_0$. Let $\mathcal{S}_{\theta}(s)=s+\theta$, and define $H^2_{\operatorname{i}}(\mathbb{C}_{\theta},\alpha)$ as the space of analytic functions in $\mathbb{C}_\theta$ such that $f \circ \mathcal{S}_\theta \circ \mathcal{T}_\alpha$ is in $H^2(\mathbb{D})$, and set
\begin{equation} \label{eq:cnorm}
	\|f\|_{H^2_{\operatorname{i}}(\mathbb{C}_\theta,\,\alpha)} := \|f\circ \mathcal{S}_\theta \circ \mathcal{T}_\alpha\|_{H^2(\mathbb{D})} = \left(\frac{\alpha}{\pi}\int_{-\infty}^\infty |f(\theta+it)|^2\,\frac{dt}{\alpha^2+t^2}\right)^{1/2}.
\end{equation}
For fixed $\theta$ and varying $\alpha$, the norms \eqref{eq:cnorm} are equivalent, but not equal. This means that the space $H^2_{\operatorname{i}}(\mathbb{C}_\theta,\,\alpha)$ does not depend on the parameter $\alpha$.

We are now ready to begin with the proof of Theorem~\ref{thm:bform}. Note that the first statement of the following lemma can be found in \cite{GH99}, but we include a short proof for the reader's benefit.

\begin{lem} \label{lem:confnorm}
	Let $\varphi \in \mathscr{G}$ with $c_0=0$ and $\varphi(+\infty)=c_1>1/2$. If $\alpha=c_1-1/2$ and $f \in \mathscr{H}^2$, then
	\begin{equation} \label{eq:confnorm}
		\|\mathscr{C}_\varphi f \|_{\mathscr{H}^2} \leq \|f\|_{H^2_{\operatorname{i}}(\mathbb{C}_{1/2},\,\alpha)}.
	\end{equation}
	Moreover, for every $c_1>1/2$ there are $\varphi\in\mathscr{G}$ with $\varphi(+\infty)=c_1$ such that equality in \eqref{eq:confnorm} is attained simultaneously for all $f \in \mathscr{H}^2$.
\end{lem}

\begin{proof}
	A computation (or \cite[Thm.~2.31]{Queffelec15}) shows that if $g \in \mathscr{H}^2$ converges uniformly in $\mathbb{C}_0$, then
	\[\|g\|_{\mathscr{H}^2} = \lim_{\beta\to \infty} \|g\|_{H^2_{\operatorname{i}}(\mathbb{C}_0,\,\beta)}.\]
	In particular, if $f$ is a Dirichlet polynomial and $\varphi$ is in $\mathscr{G}$ with $c_0=0$, then by \eqref{eq:cnorm} we get that
	\begin{equation} \label{eq:comp1}
		\|\mathscr{C}_\varphi f\|_{\mathscr{H}^2} = \lim_{\beta\to \infty} \|f \circ \varphi \|_{H^2_{\operatorname{i}}(\mathbb{C}_0,\,\beta)} = \lim_{\beta\to \infty} \|f \circ \varphi \circ \mathcal{T}_\beta\|_{H^2(\mathbb{D})}.
	\end{equation}
	Define $F \in H^2(\mathbb{D})$ and $\phi \colon \mathbb{D}\to \mathbb{D}$ by 
	\begin{align*}
		F &:= f \circ \mathcal{S}_{1/2} \circ \mathcal{T}_\alpha \\
		\phi &:= \mathcal{T}_\alpha^{-1} \circ \mathcal{S}_{1/2}^{-1} \circ \varphi \circ \mathcal{T}_\beta
	\end{align*}
	for some $0<\alpha<\infty$ to be decided later. It now follows from \eqref{eq:upperbound} that
	\begin{equation} \label{eq:comp2}
		\|f \circ \varphi \circ \mathcal{T}_\beta\|_{H^2(\mathbb{D})} = \|F \circ \phi \|_{H^2(\mathbb{D})} \leq \sqrt{\frac{1+|\phi(0)|}{1-|\phi(0)|}}\|F\|_{H^2(\mathbb{D})}.
	\end{equation}
	We compute
	\[\lim_{\beta \to \infty} \phi(0) =  \lim_{\beta\to\infty} \mathcal{T}^{-1}_\alpha(\varphi(\beta)-1/2) = \mathcal{T}^{-1}_\alpha(c_1-1/2) = 0,\]
	where the final equality is obtained by choosing $\alpha=c_1-1/2$. Combining \eqref{eq:comp1} and \eqref{eq:comp2} we find that
	\[\|\mathscr{C}_\varphi f\|_{\mathscr{H}^2} \leq \|F\|_{H^2(\mathbb{D})} = \|f\|_{H^2_{\operatorname{i}}(\mathbb{C}_{1/2},\,\alpha)}.\]
	For a general $f \in \mathscr{H}^2$, we prove \eqref{eq:confnorm} by approximating $f$ by a sequence of Dirichlet polynomials. The convergence on the right hand side is then justified by \cite[Thm.~4.11]{HLS97}.
	
	To see that $\varphi$ can be chosen to attain equality in \eqref{eq:confnorm} for every $f \in \mathscr{H}^2$, we follow an observation from \cite{BB16} (which in turn was inspired by the transference principle from \cite{QS15}) and consider the symbol defined by
	\[\varphi_\alpha(s) := \mathcal{S}_{1/2} \circ \mathcal{T}_\alpha(2^{-s}) = \frac{1}{2} + \alpha\frac{1-2^{-s}}{1+2^{-s}}.\]
	Clearly, $c_1 = 1/2 + \alpha$ as required. Observe now that the subspace $\mathscr{X}$ of $\mathscr{H}^2$ consisting of Dirichlet series of the form
	\[f(s) = \sum_{k=0}^\infty a_{2^k} 2^{-ks}\]
	is isometrically isometric to $H^2(\mathbb{D})$, through the map $2^{-s} \mapsto z$. In particular, since $\mathscr{C}_{\varphi_\alpha}$ maps $\mathscr{H}^2$ into $\mathscr{X}$, we find that
	\[\|\mathscr{C}_{\varphi_\alpha} f\|_{\mathscr{H}^2} = \|f \circ \varphi_\alpha \|_\mathscr{X} = \|f \circ \mathcal{T}_\alpha\|_{H^2(\mathbb{D})} = \|f\|_{H^2_{\operatorname{i}}(\mathbb{C}_{1/2},\,\alpha)},\]
	so equality in \eqref{eq:confnorm} is attained for $\varphi_\alpha$.
\end{proof}

We are now ready for the second part of the proof of Theorem~\ref{thm:bform}, which relies on an idea from \cite{Brevig17}.

\begin{lem} \label{lem:residue}
	For $0<\alpha<\infty$, let $C_\alpha$ denote the optimal constant in the embedding
	\[\|f\|_{H^2_{\operatorname{i}}(\mathbb{C}_{1/2},\,\alpha)} \leq C_\alpha \|f\|_{\mathscr{H}^2}.\]
	Then $C_\alpha = \sqrt{\|B_\alpha\|}$ where $B_\alpha$ is the bilinear form \eqref{eq:bform}.
\end{lem}

\begin{proof}
	Let $x>0$. As in \cite{Brevig17}, we begin by computing the integral
	\[I_\alpha(x) := \frac{\alpha}{\pi} \int_{-\infty}^\infty x^{it}\,\frac{dt}{\alpha^2+t^2} = \frac{1}{[\max(x,1/x)]^\alpha}.\]
	We insert a Dirichlet series $f(s) = \sum_{n\geq1} a_n n^{-s}$ in \eqref{eq:cnorm} with $\theta=1/2$ and compute
	\[\|f\|_{H^2_{\operatorname{i}}(\mathbb{C}_{1/2},\,\alpha)}^2 = \sum_{m=1}^\infty \sum_{n=1}^\infty \frac{a_m \overline{a_n}}{\sqrt{mn}}\,I_\alpha(n/m) =\sum_{m=1}^\infty \sum_{n=1}^\infty a_m \overline{a_n} \frac{(mn)^{\alpha-1/2}}{[\max(m,n)]^{2\alpha}}.\]
	Since the matrix associated to the bilinear form \eqref{eq:bform} is real and symmetric, it is self-adjoint. This means that the norm is attained by considering only $b = \overline{a}$. Hence we have obtained the sharp estimate
	\[\|f\|_{H^2_{\operatorname{i}}(\mathbb{C}_{1/2},\,\alpha)}^2 \leq \|B_{\alpha}\| \|f\|_{\mathscr{H}^2}^2,\]
	as desired.
\end{proof}

\begin{proof}[Final part in the proof of Theorem~\ref{thm:bform}]
	Fix $w \in \mathbb{C}_{1/2}$. The norm of $\mathscr{H}^2$ is invariant under vertical translations, so we get from Lemma~\ref{lem:confnorm} that
	\[\sup_{\substack{\varphi \in \mathscr{G} \\ \varphi(+\infty)=w}} \|\mathscr{C}_\varphi\| = \sup_{\substack{\varphi \in \mathscr{G} \\ \varphi(+\infty)=\mre(w)}} \|\mathscr{C}_\varphi\| = \sup_{\|f\|_{\mathscr{H}^2}=1} \|f\|_{H^2_{\operatorname{i}}(\mathbb{C}_{1/2},\,\alpha)}\]
	with $\alpha = \mre(w)-1/2$. We complete the proof by using Lemma~\ref{lem:residue}.
\end{proof}

\section{Continuous bilinear forms and Riemann sums} \label{sec:riemann}
As explained in the introduction, we will initiate our study of \eqref{eq:bform} by investigating the continuous version. Let $K_\alpha$ be as in \eqref{eq:Ka} and consider
\[H_\alpha(f,g) := \int_0^\infty \int_0^\infty K_\alpha(x,y) f(x)g(y)\,dydx.\]
We have the following result about the norm of $H_\alpha$ on $L^2(0,\infty)$.
\begin{thm} \label{thm:Ha}
	Let $\alpha>0$. Then $\|H_\alpha\|=2/\alpha$, that is the sharp estimate
	\[|H_\alpha(f,g)|\leq \frac{2}{\alpha}\|f\|_{L^2}\|g\|_{L^2}\]
	holds for every pair of functions $f,g \in L^2(0,\infty)$.
\end{thm}
\begin{proof}
	Following \cite[Ch.~IX]{HLP}, we apply the Cauchy--Schwarz inequality with weights $\sqrt{x/y}$ and $\sqrt{y/x}$ to find that
	\begin{align*}
		|H_\alpha(f,g)| &\leq \left(\int_0^\infty |f(x)|^2 \sqrt{x} \int_0^\infty K_\alpha(x,y)\,\frac{dy}{\sqrt{y}}\,dx\right)^{1/2}\qquad\qquad\qquad\qquad\quad \\
		&\qquad\qquad\qquad\qquad\quad \times\left(\int_0^\infty |g(y)|^2 \sqrt{y} \int_0^\infty K_\alpha(x,y)\,\frac{dx}{\sqrt{x}}\,dy\right)^{1/2}.
	\end{align*}
	We then use a substitution and the homogeneity property to conclude that an upper bound for $\|H_\alpha\|$ is
	\[\sqrt{x}\int_0^\infty K_\alpha(x,y) \frac{dy}{\sqrt{y}} = \sqrt{x} \int_0^\infty K_\alpha(x,xy) \frac{x \,dy}{\sqrt{xy}} = \int_0^\infty K_\alpha(1,y)\,\frac{dy}{\sqrt{y}}=:C_\alpha.\]
	We then easily compute
	\begin{equation} \label{eq:Rint}
		C_\alpha = \int_0^1 y^{\alpha-1}\, dy + \int_1^\infty y^{-\alpha-1}\,dy = \frac{1}{\alpha} + \frac{1}{\alpha} = \frac{2}{\alpha}.
	\end{equation}
	To prove optimality, let $0<\varepsilon<\alpha$ and set 
	\[f(t) = g(t) = \begin{cases}
		0, & t \in(0,1), \\
		t^{-1/2-\varepsilon}, & t\in(1,\infty).
	\end{cases}\]
	A direct computation gives that
	\[H_\alpha(f,g) = \int_1^\infty \int_1^\infty \frac{(xy)^{\alpha-1-\varepsilon}}{[\max(x,y)]^{2\alpha}}\,dydx =\left(\frac{1}{\alpha-\varepsilon}+\frac{1}{\alpha+\varepsilon}\right)\|f\|_{L^2}^2+ O(1).\]
	Clearly $\|f\|_{L^2} \to \infty$ as $\varepsilon\to 0^+$, so we find that
	\[\|H_\alpha\| \geq \lim_{\varepsilon\to 0^+}  \left(\frac{1}{\alpha-\varepsilon}+\frac{1}{\alpha+\varepsilon}\right) = \frac{2}{\alpha}. \qedhere\]
\end{proof}
We proceed by showing that the lower bound from the continuous setting carries across to $\|B_\alpha\|$.
\begin{lem} \label{lem:2alower}
	For $0<\alpha<\infty$, we have that $\|B_\alpha\|\geq 2/\alpha$.
\end{lem}
\begin{proof}
	Set $a_m = m^{-1/2-\varepsilon}$ and $b_n = n^{-1/2-\varepsilon}$ for some $0<\varepsilon<\alpha$. We get
	\begin{equation} \label{eq:Ba2a}
		\|B_\alpha\| \geq \frac{1}{\zeta(1+2\varepsilon)} \sum_{m=1}^\infty \sum_{n=1}^\infty \frac{(mn)^{\alpha-1-\varepsilon}}{[\max(m,n)]^{2\alpha}} =: \frac{D_\alpha}{\zeta(1+2\varepsilon)}.
	\end{equation}
	Now, 
	\[D_\alpha = \sum_{m=1}^\infty m^{-\alpha-1-\varepsilon} \sum_{n=1}^m n^{\alpha-1-\varepsilon} + m^{\alpha-1-\varepsilon} \sum_{n=m+1}^\infty n^{-\alpha-1-\varepsilon}.\]
	Standard computations shows that
	\begin{align*}
		\sum_{n=1}^m n^{\alpha-1-\varepsilon} &= \frac{n^{\alpha-\varepsilon}}{\alpha-\varepsilon} + O(n^{\alpha-1}), \\
		\sum_{n=m+1}^\infty n^{-\alpha-1-\varepsilon} &= \frac{n^{-\alpha-\varepsilon}}{\alpha+\varepsilon} + O(n^{-\alpha-1}),
	\end{align*}
	so we get that
	\[D_\alpha = \sum_{m=1}^\infty \sum_{n=1}^\infty \frac{(mn)^{\alpha-1-\varepsilon}}{[\max(m,n)]^{2\alpha}} = \zeta(1+2\varepsilon) \left(\frac{1}{\alpha-\varepsilon} + \frac{1}{\alpha+\varepsilon}\right) + O(1).\]
	We insert this into \eqref{eq:Ba2a} and let $\varepsilon\to 0^+$ to complete the proof.
\end{proof}
To investigate upper bounds for $\|B_\alpha\|$, we use the same weighted Cauchy--Schwarz inequality as in the proof of Theorem~\ref{thm:Ha} to get 
\begin{align*}
	|B_\alpha(a,b)| &\leq \left(\sum_{m=1}^\infty |a_m|^2 \sqrt{m}\sum_{n=1}^\infty \frac{(mn)^{\alpha-1/2}}{[\max(m,n)]^{2\alpha}}\,\sqrt{\frac{1}{n}}\right)^{1/2} \qquad\qquad\qquad\qquad \\
	&\qquad\qquad\qquad\qquad \times \left(\sum_{n=1}^\infty |b_n|^2 \sqrt{n}\sum_{m=1}^\infty \frac{(mn)^{\alpha-1/2}}{[\max(m,n)]^{2\alpha}}\,\sqrt{\frac{1}{m}}\right)^{1/2}.
\end{align*}
By symmetry, we obtain the upper bound $\|B_\alpha\| \leq \sup_{m} S_\alpha(m)$, where 
\begin{equation} \label{eq:Sxi}
	S_\alpha(m) := \sqrt{m} \sum_{n=1}^\infty \frac{(mn)^{\alpha-1/2}}{[\max(m,n)]^{2\alpha}}\,\sqrt{\frac{1}{n}}.
\end{equation}
Our next goal is to prove the following result, which by the preceding discussion constitutes the upper bound in Theorem~\ref{thm:Balpha}.
\begin{lem} \label{lem:Sasup}
	Let $0<\alpha<\infty$ and let $S_\alpha$ be as in \eqref{eq:Sxi}. Then
	\[\sup_{m} S_\alpha(m) = \max\left(\frac{2}{\alpha}\,,\,\zeta(1+\alpha)\right).\]
\end{lem}
The first step towards the proof of Lemma~\ref{lem:Sasup} is to rewrite \eqref{eq:Sxi} as  
\begin{equation} \label{eq:riemannsum}
	\begin{split}
		S_\alpha(m) &= m^{-\alpha}\sum_{n=1}^m n^{\alpha-1} + m^\alpha \sum_{n=m+1}^\infty n^{-\alpha-1}\\
		&= \frac{1}{m}\sum_{n=1}^m \left(0+\frac{n}{m}\right)^{\alpha-1} + \sum_{j=1}^\infty \frac{1}{m}\sum_{n=1}^m \left(j+\frac{n}{m}\right)^{-\alpha-1},
	\end{split}
\end{equation}
to see that $S_\alpha(m)$ is a Riemann sum of the two integrals in \eqref{eq:Rint} with step length $m^{-1}$, taking the value at the upper endpoint for each interval. Hence we conclude that
\begin{equation} \label{eq:limit}
	\lim_{m \to \infty } S_\alpha(m) = \frac{2}{\alpha}.
\end{equation}
Note also that $S_\alpha(1)=\zeta(1+\alpha)$. Hence Lemma~\ref{lem:Sasup} states that the first or the ``last'' element of the sequence $S_\alpha$ is always the biggest. The fact that $S_\alpha(m)$ are Riemann sums of \eqref{eq:Rint} directly gives the following simple proof.

\begin{proof}[Proof of Lemma~\ref{lem:Sasup}: $0<\alpha\leq 1$]
	If $0<\alpha\leq1$, then $y \mapsto y^{\alpha-1}$ and $y \mapsto y^{-\alpha-1}$ are decreasing functions on $(0,1)$ and $(1,\infty)$, respectively. This means that both sums in \eqref{eq:riemannsum} are increasing sequences, and the limit \eqref{eq:limit} is also the supremum of the combined sequence. 
\end{proof}

Note that when $1<\alpha<\infty$, the function $y\mapsto y^{\alpha-1}$ is increasing on $(0,1)$ so we can no longer take the limit and obtain the supremum. This is the phase change mentioned in the introduction. Nevertheless, when $\alpha$ is large enough, we may conclude by rather savage estimates.

\begin{proof}[Proof of Lemma~\ref{lem:Sasup}: $3 \leq \alpha < \infty$]
	Consider the sums in \eqref{eq:riemannsum}. We know that if $\alpha>1$, then the first sum is decreasing and the second sum is increasing. In particular, if $m\geq2$, then
	\[S_\alpha(m) \leq 2^{-\alpha}\left(1+2^{\alpha-1}\right) + \frac{1}{\alpha} =  2^{-\alpha} + {1/2} + \frac{1}{\alpha}.\]
	Furthermore, $S_\alpha(1) = \zeta(1+\alpha) \geq 1 + 2^{-\alpha-1}$, so we get that $S_\alpha(1)\geq S_\alpha(m)$ whenever
	\[1 + 2^{-\alpha-1} \geq 2^{-\alpha} + {1/2} + \frac{1}{\alpha} \qquad \Longleftrightarrow \qquad {1/2}-\frac{1}{\alpha}-2^{-\alpha-1} \geq0.\]
	This final expression is clearly increasing in $\alpha$ and positive for $\alpha = 3$, so we conclude that $\sup_m S_\alpha(m) = \zeta(1+\alpha)$ when $\alpha\geq3$.
\end{proof}

\section{Proof of Lemma~\ref{lem:Sasup}: $1 \leq \alpha \leq 3$} \label{sec:lemma}
For fixed $\alpha>1$, we do not know if the sequence $S_\alpha(m)$ is increasing or decreasing, since it is the sum of one increasing and one decreasing sequence. Our general approach is therefore to obtain decreasing upper bounds for $S_\alpha(m)$ when $m\geq2$, which we then compare with $2/\alpha$ and $\zeta(1+\alpha)$.
  
To obtain these estimates, we will apply the Euler--Maclaurin summation formula (see~\cite[Ch.~B]{MV07}). In preparation, let us recall a few properties of Bernoulli polynomials, denoted $B_k(x)$. We will only have use of the first five polynomials, which are
\[B_1(x) = x-\frac{1}{2},\qquad B_2(x) = x^2-x+\frac{1}{6},\qquad B_3(x) = x^3 - \frac{3}{2}x^2+\frac{1}{2}x,\]
\[B_4(x) = x^4 -2x^3+x^2-\frac{1}{30},\qquad B_5(x) = x^5 - \frac{5}{2}x^4+\frac{5}{3}x^3-\frac{1}{6}x.\]
To analyse the remainder terms in the Euler--Maclaurin summation formula, we make use of the following simple result.
\begin{lem} \label{lem:symmetry}
	If $g$ is positive, continuous and decreasing on $[0,1]$, then
	\[\operatorname{sign}\left(\int_0^1 g(x)B_{2k+1}(x) \,dx\right) = (-1)^{k-1}.\]
\end{lem}
\begin{proof}
	The following well-known facts can be found in \cite[Thm.~B.1]{MV07}:
	\begin{align*}
		B_{2k+1}(x) &= -B_{2k+1}(1-x)  & & (0<x<1),  \\
		\operatorname{sign}\big(B_{2k+1}(x)\big) &= (-1)^{k-1} & & (0<x<1/2).
	\end{align*}
	The statement now follows from a symmetry consideration.
\end{proof}
The Bernoulli numbers are defined by $B_k := B_k(0)$. We will only use that $B_2=1/6$ and $B_4=-1/30$. Let $\{x\}$ denote the fractional part of $x$. Suppose that $f \in C^{2k+1}\big([m,\infty)\big)$. For $m\geq1$ and $k\geq0$ we have that
\begin{equation} \label{eq:euler1}
	\sum_{n=m+1}^\infty f(n) = \int_m^\infty f(x)\,dx - \frac{f(m)}{2} - \sum_{j=1}^k \frac{B_{2j}}{(2j)!}f^{(2j-1)}(m) + R_k(m).
\end{equation}
The remainder term in \eqref{eq:euler1} is given by 
\[R_k(m) = \frac{1}{(2k+1)!}\int_m^\infty f^{(2k+1)}(x) B_{2k+1}(\{x\})\,dx.\]
Similarly to \eqref{eq:euler1}, if $f \in C^{(2k+1)}\big([1,m]\big)$ it holds for $m\geq1$ and $k\geq0$ that
\begin{equation} \label{eq:euler2}
	\begin{split}
		\sum_{n=1}^m f(n) &= \int_1^m f(x)\,dx + \frac{f(m)+f(1)}{2} \\ &+ \sum_{j=1}^k \frac{B_{2j}}{(2j)!}\left(f^{(2j-1)}(m)-f^{(2j-1)}(1)\right) + \widetilde{R}_k(m).
	\end{split}
\end{equation}
The remainder term in \eqref{eq:euler2} is given by 
\[\widetilde{R}_k(m) = \frac{1}{(2k+1)!}\int_1^m f^{(2k+1)}(x) B_{2k+1}(\{x\})\,dx.\]
We are now ready to obtain our estimates.

\begin{lem} \label{lem:euler}
	Let $0<\alpha<\infty$. Then
	\begin{align}
		m^\alpha \sum_{n=m+1}^\infty n^{-\alpha-1} &\leq \frac{1}{\alpha}-\frac{1}{2m} + \frac{(\alpha+1)}{12m^2}, \label{eq:tailupper} \\
		\zeta(1+\alpha) &\geq \frac{1}{\alpha} + \frac{1}{2} + \frac{(\alpha+1)}{12} - \frac{(\alpha+1)(\alpha+2)(\alpha+3)}{720}. \label{eq:zetalower}
		\intertext{If $1\leq \alpha \leq 2$, then}
		\frac{1}{m^\alpha} \sum_{n=1}^m n^{\alpha-1} &\leq \frac{1}{\alpha} + \frac{1}{2m} + \frac{(\alpha-1)}{12m^2} - \frac{(\alpha-3)(\alpha-4)}{12\alpha}\frac{1}{m^\alpha}, \label{eq:upper12}
		\intertext{and if $2\leq \alpha \leq 3$, then}
		\frac{1}{m^\alpha} \sum_{n=1}^m n^{\alpha-1} &\leq \frac{1}{\alpha} + \frac{1}{2m} + \frac{(\alpha-1)}{12m^2}. \label{eq:upper23}
	\end{align} 
\end{lem}
\begin{proof}
	To get \eqref{eq:tailupper} and \eqref{eq:zetalower} we apply \eqref{eq:euler1} to $f(x) = x^{-\alpha-1}$ with $k=1$ and $k=2$, respectively. Note that for \eqref{eq:zetalower} we take $m=1$. To verify the sign of the remainder term, we appeal to Lemma~\ref{lem:symmetry} for $g(x) = -f^{(2k+1)}(x)$. 
	
	To prove \eqref{eq:upper12}, we use \eqref{eq:euler2} with $f(x) = x^{\alpha-1}$ and $k=2$. To see that the remainder term is negative, we note that $g(x) = f^{(5)}(x)$ is positive and decreasing when $1<\alpha<2$ and use Lemma~\ref{lem:symmetry}. Hence we get that
	\begin{align*}
		m^{-\alpha} \sum_{n=1}^m n^{\alpha-1} &\leq \frac{1}{\alpha} + \frac{1}{2m} + \frac{(\alpha-1)}{12}\frac{1}{m^2} - \frac{(\alpha-1)(\alpha-2)}{720}\left(\frac{1}{m^4}-\frac{1}{m^\alpha}\right) \\
		&\qquad\qquad\qquad\qquad\qquad\qquad\qquad+ \left(-\frac{1}{\alpha}+\frac{1}{2}-\frac{(\alpha-1)}{12}\right)\frac{1}{m^\alpha} \\
		&\leq \frac{1}{\alpha} + \frac{1}{2m} + \frac{(\alpha-1)}{12}\frac{1}{m^2} - \frac{(\alpha-3)(\alpha-4)}{12\alpha}\frac{1}{m^\alpha},
	\end{align*}	
	where we in the second inequality used (twice) that $1\leq \alpha \leq 2$ to conclude that the fourth term is negative.
	
	Finally, for \eqref{eq:upper23}, we again use \eqref{eq:euler2} with $f(x) = x^{\alpha-1}$ and $k=1$. The remainder term is negative by Lemma~\ref{lem:symmetry}, since $g(x) = -f^{(3)}(x)$ is positive and decreasing when $2 < \alpha < 3$. Hence
	\begin{align*}
		\frac{1}{m^\alpha} \sum_{n=1}^m &\leq \frac{1}{\alpha} + \frac{1}{2m} + \frac{(\alpha-1)}{12}\frac{1}{m^2} + \left(-\frac{1}{\alpha}+\frac{1}{2} - \frac{(\alpha-1)}{12}\right)\frac{1}{m^\alpha}.
	\end{align*}
	The factor in front of $m^{-\alpha}$ is negative when $2< \alpha < 3$.
\end{proof}

\begin{proof}[Proof of Lemma~\ref{lem:Sasup}: $2\leq \alpha \leq 3$]
	We want to prove that $S_\alpha(1)\geq S_\alpha(m)$. We combine \eqref{eq:tailupper} and \eqref{eq:upper23} to get that if $m\geq2$, then
	\[S_\alpha(m) \leq \frac{2}{\alpha} + \frac{\alpha}{6m^2} \leq \frac{2}{\alpha} + \frac{\alpha}{24}.\]
	Recall that $S_\alpha(1)=\zeta(1+\alpha)$. By \eqref{eq:zetalower} and the fact that $2\leq \alpha \leq 3$, we get
	\[\zeta(1+\alpha) \geq \frac{1}{\alpha} + \frac{1}{2} + \frac{(\alpha+1)}{12} - \frac{(\alpha+1)(\alpha+2)(\alpha+3)}{720} \geq \frac{1}{\alpha} + \frac{2}{3}.\]
	We complete the proof by checking that 
	\[\frac{2}{\alpha} + \frac{\alpha}{24} < \frac{1}{\alpha} + \frac{2}{3}\]
	for $2\leq \alpha \leq 3$.	
\end{proof}

The proof of the following lemma is a straightforward calculus argument, but it is very tedious and therefore omitted.
\begin{lem} \label{lem:auxeq}
	For $1\leq \alpha \leq 2$, consider the following functions
	\begin{align}
		h_1(\alpha) &:= \frac{(\alpha-3)(\alpha-4)}{12\alpha}\frac{1}{2^\alpha}-\frac{\alpha}{24}, \label{eq:h1}\\
		h_2(\alpha) &:= \frac{1}{2} + \frac{(\alpha+1)}{12} - \frac{(\alpha+1)(\alpha+2)(\alpha+3)}{720} - \frac{1}{\alpha} + h_1(\alpha). \label{eq:h2}
	\end{align}
	Then 
	\begin{itemize}
		\item $h_1$ is strictly decreasing on $[1,2]$ and the equation $h_1(\alpha)=0$ has the unique solution $\alpha_1=1.553\ldots$,
		\item $h_2$ is strictly increasing on $[1,2]$ and the equation $h_2(\alpha)=0$ has the unique solution $\alpha_2=1.507\ldots$,
	\end{itemize}
	and in particular, $\alpha_1 > \alpha_2$.
\end{lem}

\begin{proof}[Proof of Lemma~\ref{lem:Sasup}: $1\leq\alpha\leq2$]
	We combine \eqref{eq:tailupper} and \eqref{eq:upper12} to obtain
	\[S_\alpha(m) \leq \frac{2}{\alpha} + \frac{\alpha}{6}\frac{1}{m^2}- \frac{(\alpha-3)(\alpha-4)}{12\alpha}\frac{1}{m^\alpha}.\]
	We check that the right hand side is decreasing in $m\geq2$ for $1\leq\alpha\leq2$ to conclude that
	\begin{equation} \label{eq:Sam12}
		S_\alpha(m) \leq \frac{2}{\alpha} + \frac{\alpha}{24}-\frac{(\alpha-3)(\alpha-4)}{12\alpha} \frac{1}{2^\alpha}.
	\end{equation}
	We then use \eqref{eq:Sam12} and \eqref{eq:zetalower} to obtain
	\begin{align*}
		\frac{2}{\alpha} - S_\alpha(m) \geq h_1(\alpha), \\
		\zeta(1+\alpha) - S_\alpha(m) \geq h_2(\alpha),
	\end{align*}
	where $h_1$ and $h_2$ are the functions from \eqref{eq:h1} and \eqref{eq:h2}, respectively. By Lemma~\ref{lem:auxeq} we can therefore conclude that
	\[\max\left(\frac{2}{\alpha}\,,\,\zeta(1+\alpha)\right) \geq S_\alpha(m). \qedhere\]
\end{proof}

\section{Concluding remarks} \label{sec:remarks}
\subsection{} Our first remarks concern the relationship between the upper bound for composition operators on $H^2(\mathbb{D})$ from \eqref{eq:upperbound}, Theorem~\ref{thm:bform} and Corollary~\ref{cor:Cphileq}. We first observe that the behaviour of the upper bound \eqref{eq:Cphileq} as $\varphi(+\infty)$ approaches the boundary of $\mathbb{C}_{1/2}$ is identical to the behaviour of the upper bound \eqref{eq:upperbound} as $\phi(0)$ approaches the boundary of $\mathbb{D}$, since $1+|\phi(0)|\to2$. 

Let us next discuss the transference principle from \cite{QS15}. Suppose that $\phi$ is the symbol of a composition operator on $H^2(\mathbb{D})$. The composition operator $\mathscr{C}_\phi\colon H^2(\mathbb{D})\to H^2(\mathbb{D})$ can be transferred to an operator on $\mathscr{H}^2$ with symbol
\begin{equation} \label{eq:transference}
	\varphi_\alpha := \mathcal{S}_{1/2} \circ \mathcal{T}_\alpha \circ \phi \circ \mathcal{I},
\end{equation}
where $\mathcal{T}_\alpha$ is as in \eqref{eq:Ta} and $\mathcal{I}(s) := 2^{-s}$. In the second part of the proof of Lemma~\ref{lem:confnorm}, we essentially transfer $\phi(z)=z$ in this way. Note that in \cite{QS15} only $\alpha=1$ is considered, but that particular choice of $\alpha$ is not important for their considerations. We can obtain the following result about the norm of the transferred composition operator.

\begin{thm} \label{thm:transference}
	Fix $0<\alpha<\infty$. Suppose that $\phi \colon \mathbb{D} \to \mathbb{D}$ is analytic, and set $r=|\phi(0)|<1$. Let
	\[\alpha_r := \alpha\frac{1-r}{1+r}.\]
	Define $\varphi_\alpha$ as in \eqref{eq:transference}. The upper bound $\|\mathscr{C}_{\varphi_\alpha}\|\leq \sqrt{\|B_{\alpha_r}\|}$ is sharp.
\end{thm}
\begin{proof}
	The proof is similar to that of Theorem~\ref{thm:bform}. We begin by noting that
	\[\|\mathscr{C}_{\varphi_\alpha}\|_{\mathscr{H}^2} = \|f \circ \mathcal{S}_{1/2} \circ \mathcal{T}_\alpha \circ \phi\|_{H^2(\mathbb{D})}.\]
	After a rotation, we may assume that $\phi(0)=r$. Let $\phi_r(z) := (r-z)/(1-rz)$. It now follows from \eqref{eq:upperbound} that
	\begin{equation} \label{eq:transest}
		\|f \circ \mathcal{S}_{1/2} \circ \mathcal{T}_\alpha \circ \phi_r \circ\phi_r^{-1} \circ \phi\|_{H^2(\mathbb{D})} \leq \|f\circ \mathcal{S}_{1/2} \circ \mathcal{T}_\alpha \circ \phi_r\|_{H^2(\mathbb{D})},
	\end{equation}
	since $(\phi_r^{-1}\circ\phi)(0)=0$. The proof is completed by computing $\mathcal{T}_\alpha \circ \phi_r = \mathcal{T}_{\alpha_r}$, then using \eqref{eq:cnorm} and Lemma~\ref{lem:residue}. Equality in \eqref{eq:transest} is attained for $\phi=\phi_r$.
\end{proof}
Note that when $0<\alpha\leq \alpha_0$, we can combine Theorem~\ref{thm:Balpha} and \eqref{eq:upperbound} to restate the sharp upper bound $\|\mathscr{C}_{\varphi_\alpha}\|\leq \sqrt{\|B_{\alpha_r}\|}$ as
\begin{equation} \label{eq:restateest}
	\|\mathscr {C}_{\varphi_\alpha}\|_{\mathscr{H}^2\to\mathscr{H}^2} \leq \sqrt{\|B_\alpha\|}\cdot \|C_\varphi\|_{H^2(\mathbb{D})\to H^2(\mathbb{D})}.
\end{equation}
The upper bound \eqref{eq:restateest} in fact holds for every $\alpha>0$. This can be deduced from the arguments in \cite[Sec.~9]{QS15} and Theorem~\ref{thm:bform}. However, \eqref{eq:restateest} cannot be sharp for every $\alpha>\alpha_0$ unless $r=0$. To see this, consider $\phi_r$ for some fixed $0<r<1$, then appeal to Theorem~\ref{thm:Balpha} and choose a large $\alpha$ such that
\[\frac{\zeta(1+\alpha_r)}{\zeta(1+2\alpha)}<\frac{1+r}{1-r}.\]

\subsection{} Let us now discuss Theorem~\ref{thm:Balpha} for $\alpha>\alpha_0$. We began our analysis of $B_\alpha$ with the application of the Cauchy--Schwarz inequality to obtain \eqref{eq:Sxi}. To obtain an upper bound for $\|B_\alpha\|$, we computed the supremum of the sequence $S_\alpha(m)$. Note that $a_m = m^{-1/2-\varepsilon}$ and $b_n = n^{-1/2-\varepsilon}$ which gives the lower bound $2/\alpha$ in Lemma~\ref{lem:2alower} is chosen to attain equality in the weighted Cauchy--Schwarz inequality. We get that $S_\alpha(+\infty) =2/\alpha$ is also an upper bound when $\alpha\leq\alpha_0$, since the ``tail estimate'' dominates $S_\alpha(m)$ for all $m$. 

We cannot expect the upper bound $S_\alpha(1)=\zeta(1+\alpha)$ to be attained in the same way, since the supremum is attained in the first summand of \eqref{eq:Sxi}. Hence we conjecture that $\|B_\alpha\|\leq \zeta(1+\alpha)$ is not sharp for all $\alpha>\alpha_0$. 

Our main effort has been directed at the upper bound in Theorem~\ref{thm:Balpha}. It is easy to improve the lower bound coming from the point estimate \eqref{eq:peval}. We offer only the following example result in this direction. If $\alpha>1$, then 
\begin{equation} \label{eq:lowernew}
	\|B_\alpha\| \geq 2 - \frac{\zeta(2\alpha)}{\zeta(2\alpha-1)}.
\end{equation}
To prove \eqref{eq:lowernew}, set $a_m = m^{-\alpha+1/2}$ and $b_n = n^{-\alpha+1/2}$. The estimate follows at once from the computation
\[\sum_{m=1}^\infty \sum_{n=1}^\infty \frac{1}{[\max(m,n)]^{2\alpha}} = 2\zeta(2\alpha-1)-\zeta(2\alpha).\]
Comparing with \eqref{eq:lowernew}, we find that the upper bound $\|B_\alpha\|\leq 2/\alpha$ cannot hold for $\alpha\geq1.7$. Note also that the difference between the lower bounds from Theorem~\ref{thm:Balpha} and \eqref{eq:lowernew} is irrelevant for large $\alpha$, compared to $\zeta(1+\alpha)$. In combination, these observations lead to the following questions. 
\begin{enumerate}
	\item[(a)] For which $1.48\ldots = \alpha_0 \leq \alpha <1.7$ does the upper bound $\|B_\alpha\|\leq 2/\alpha$ cease to hold?
	\item[(b)] What is the asymptotic decay of $\|B_\alpha\|-1$ as $\alpha\to\infty$?
\end{enumerate}
For question (a), we suggest investigating the attractive special case $\alpha=3/2$, which can be formulated as follows. Let $a=(a_1,a_2,\ldots)$ be a non-negative sequence. Find the best constant $C$ in the inequality
\[\sum_{m=1}^\infty \sum_{n=1}^\infty a_m a_n \frac{mn}{[\max(m,n)]^3} \leq C \sum_{m=1}^\infty a_m^2.\]
We know from Theorem~\ref{thm:Balpha} that $1.33\ldots = 4/3 \leq C \leq \zeta(5/2)=1.34\ldots$.

To make question (b) precise, we deduce from Theorem~\ref{thm:Balpha} that there are positive constants $C_1$ and $C_2$ such that for $\alpha\geq2$, we have 
\[C_1 4^{-\alpha} \leq \|B_\alpha\|-1 \leq C_2 2^{-\alpha}.\]
Note that \eqref{eq:lowernew} only changes the constant $C_1$. It would be interesting to decide which (if either) of these bounds is of the correct order. In analogy with \eqref{eq:upperbound} and \eqref{eq:lowerbound}, one might conjecture that $4^{-\alpha}$ is correct.

\bibliographystyle{amsplain} 
\bibliography{l}
\end{document}